\documentclass[12pt,final]{article}
\usepackage[margin=2.5cm,top=1.5cm]{geometry}
\usepackage[all,arc]{xy}
\usepackage{mathrsfs}
\usepackage{amsmath,amsthm,amssymb,enumerate}
\usepackage{color}

\newtheorem{definition}{Definition}[section]

\newtheorem{theorem}[definition]{Theorem}

\newtheorem{proposition}[definition]{Proposition}

\newtheorem{example}[definition]{Example}
\newtheorem{conjecture}[definition]{Conjecture}
\newtheorem{problem}[definition]{Question}

\newtheorem{remark}[definition]{Remark}

\newcommand{\Romannum}[1]{\uppercase\expandafter{\romannumeral #1}}

\numberwithin{equation}{section}

\newcommand\keywordsname{Key words}
\newcommand\AMSname{AMS subject classifications}

\newenvironment{@abssec}[1]{%
     \if@twocolumn
       \section*{#1}%
     \else
       \vspace{.05in}\footnotesize
       \parindent .2in
         {\upshape\bfseries #1. }\ignorespaces
     \fi}
     {\if@twocolumn\else\par\vspace{.1in}\fi}

\begin{document}

\title{New result and some open problems on the primitive degree of nonnegative  tensors
\footnote{P. Yuan's research is supported by the NSF of China (Grant No. 11271142) and
the Guangdong Provincial Natural Science Foundation(Grant No. S2012010009942),
L. You's research is supported by the Zhujiang Technology New Star Foundation of
Guangzhou (Grant No. 2011J2200090) and Program on International Cooperation and Innovation, Department of Education,
Guangdong Province (Grant No.2012gjhz0007).}}
\author{Pingzhi Yuan \footnote{{\it{Corresponding author:\;}}yuanpz@scnu.edu.cn.}
 \qquad Zilong He\footnote{{\it{Email address:\;}}hzldew@qq.com.}
 \qquad Lihua You\footnote{{\it{Email address:\;}}ylhua@scnu.edu.cn.}
 }
\vskip.2cm
\date{{\small
School of Mathematical Sciences, South China Normal University,\\
Guangzhou, 510631, P.R. China\\
}}
\maketitle

\begin{abstract} In this paper, we show that the exponent set of nonnegative primitive tensors with order $m(\ge3)$ and dimension $n$
is $\{1,2,\ldots, (n-1)^2+1\}, $
and  propose some open problems for further  research.

\vskip.2cm \noindent{\it{AMS classification:}} 05C50;  15A69
 \vskip.2cm \noindent{\it{Keywords:}}  tensor; the exponent set; primitive tensor; primitive degree.
\end{abstract}

\section{Introduction and a survey}
\hskip.6cm A  nonnegative square matrix $A= (a_{ij})$ of order $n$ is nonnegative primitive (or simply, primitive) if $A^k>0$ for
some positive integer $k$. The least such $k$ is called the primitive exponent (or simply, exponent) of $A$ and is
denoted by $\gamma(A)$.

Since the work of Qi \cite{Qi05} and Lim \cite{Li05}, the study  of tensors which regarded as the generalization of matrices,
the spectra of tensors (and hypergraphs) and their various applications has attracted much attention and interest.

As is in  \cite{Qi05}, an order $m$ dimension $n$ tensor $\mathbb{A}= (a_{i_1i_2\ldots i_m})_{1\le i_j\le n \hskip.2cm (j=1, \ldots, m)}$ over the complex field $\mathbb{C}$ is a multidimensional array with all entries $a_{i_1i_2\ldots i_m}\in\mathbb{C}\, ( i_1, \ldots, i_m\in [n]=\{1, \ldots, n\})$.
In \cite{Ch1} and \cite{Ch2}, Chang et al investigated the properties of the spectra of nonnegative tensors,  defined the irreducibility of tensors and the primitivity of nonnegative tensors (as Definition \ref{defn11}), and extended many important properties of primitive matrices to primitive tensors.


\begin{definition} {\rm (\cite{Ch2, Pe102}) \label{defn11}} Let $\mathbb{A}$ be a nonnegative  tensor with order $m$ and dimension $n$,
$x=(x_1, x_2, \ldots, x_n)^T\in\mathbb{R}^n$ a vector and $x^{[r]}=(x_1^r, x_2^r, \ldots, x_n^r)^T$. Define the map $T_\mathbb{A}$ from $\mathbb{R}^n$ to $\mathbb{R}^n$ as: $T_\mathbb{A}(x)=(\mathbb{A}x)^{[\frac{1}{m-1}]}$. If there exists some positive integer $r$ such that $T_\mathbb{A}^r(x)>0$ for all nonnegative nonzero vectors $x\in\mathbb{R}^n$, then $\mathbb{A}$ is called primitive and the smallest such integer $r$ is called the primitive degree of $\mathbb{A}$, denoted by $\gamma(\mathbb{A})$. \end{definition}

Recently, Shao \cite{Sh12} defined the general product of two n-dimensional tensors as follows.
\begin{definition}\label{defn12}
Let $\mathbb{A}$ {\rm (}and $\mathbb{B}${\rm)} be an order $m\ge2$ {\rm (}and $k\ge 1${\rm)}, dimension $n$ tensor, respectively.
Define the general product  $\mathbb{A}\mathbb{B}$ to be the following tensor $\mathbb{D}$ of order $(m-1)(k-1)+1$ and dimension $n$:
$$ d_{i\alpha_1\ldots\alpha_{m-1}}=\sum\limits_{i_2, \ldots, i_m=1}^na_{ii_2\ldots i_m}b_{i_2\alpha_1}\ldots b_{i_m\alpha_{m-1}} \quad (i\in[n], \, \alpha_1, \ldots, \alpha_{m-1}\in[n]^{k-1}).$$
\end{definition}

The tensor product  is a generalization of the usual matrix product, and satisfies  a very useful property: the associative law (\cite{Sh12}, Theorem 1.1).
With the general product,  when $k=1$ and $\mathbb{B}=x=(x_1,\ldots, x_n)^T\in \mathbb{C}^n$ is a vector of dimension $n$,
then $\mathbb{A}\mathbb{B} = \mathbb{A}x$ is still a vector of dimension $n$, and for any $i\in [n],$
$$(\mathbb{A}\mathbb{B})_i=(\mathbb{A}x)_i=\sum\limits_{i_2, \ldots, i_m=1}^na_{ii_2\ldots i_m}x_{i_2}\ldots x_{i_m}.$$

As an application of the general tensor product defined by Shao \cite{Sh12},
Shao presented a simple characterization of the primitive tensors.
Now we  give the definition of ``essentially positive" which introduced by Pearson.

\begin{definition} {\rm (\cite{Pe10}, Definition 3.1)} \label{defn13}A nonnegative tensor $\mathbb{A}$ is  called essentially positive,
if for any nonnegative nonzero vector $x\in \mathbb{R}^n, \mathbb{A}x>0$ holds.
\end{definition}

\begin{proposition}{\rm(\cite{Sh12}, Proposition 4.1)} \label{pro14}
Let $\mathbb{A}$ be an order $m$ and dimension $n$ nonnegative  tensor. Then the following three conditions are equivalent:

{\rm (i). } For any $i, j\in[n], a_{ij\cdots j}>0$ holds.

{\rm (ii). } For any $j\in[n],  \mathbb{A}e_j>0$ holds {\rm(}where $e_j$ is the $j^{th}$ column of the identity matrix $I_n${\rm)}.

{\rm (iii). } For any nonnegative nonzero vector $x\in \mathbb{R}^n, \mathbb{A}x>0$ holds. \end{proposition}

By Proposition \ref{pro14}, the following Definition \ref{defn15} is equivalent to Definition \ref{defn13}.

\begin{definition} {\rm (\cite{Sh12}, Definition 4.1)} \label{defn15}
A nonnegative tensor $\mathbb{A}$ is  called essentially positive, if it satisfies one of the three conditions in Proposition \ref{pro14}.\end{definition}

Let $Z(\mathbb{A})$ be the tensor obtained by replacing all the nonzero entries of $\mathbb{A}$ by one.
Then  $Z(\mathbb{A})$ is called the zero-nonzero pattern of A (or simply the zero pattern of $\mathbb{A}$).
In \cite{Sh12}, Shao showed the following characterization and defined the primitive degree by using
the properties of tensor product and the zero patterns.

\begin{proposition}\label{pro16}{\rm (\cite{Sh12}, Theorem 4.1)}
A nonnegative tensor $\mathbb{A}$ is primitive if and only if there exists some positive integer $r$ such that $\mathbb{A}^r$ is essentially positive. Furthermore, the smallest such $r$ is the primitive degree of $\mathbb{A}$, $\gamma(\mathbb{A})$.
\end{proposition}

The concept of the majorization matrix of a tensor introduced by Pearson is very useful.

\begin{definition}\label{defn21}{\rm (\cite{Pe10},  Definition 2.1)}
 The majorization matrix  $M(\mathbb{A})$ of the tensor $\mathbb{A}$ is defined as  $(M(\mathbb{A}))_{ij}=
a_{ij\ldots j}, i, j\in[n]$.
\end{definition}

By  Definition \ref{defn15}, Proposition \ref{pro16} and  Definition \ref{defn21},
 the following characterization of the primitive tensors was easily obtained.

\begin{proposition}\label{pro22}{\rm (\cite{YHY13}, Remark 2.6)}
Let $\mathbb{A}$ be a nonnegative  tensor with order $m$ and dimension $n$.
Then $\mathbb{A}$ is primitive if and only if there exists some positive integer $r$ such that $M(\mathbb{A}^r)>0.$
Furthermore, the smallest such $r$ is the primitive degree of $\mathbb{A}$, $\gamma(\mathbb{A})$.
\end{proposition}

On the primitive degree $\gamma(\mathbb{A})$, Shao proposed  the following conjecture for further research.%

\begin{conjecture} \label{con17}{\rm (\cite{Sh12},  Conjecture 1)}
When $m$ is fixed, then there exists some polynomial $f(n)$ on $n$ such that $\gamma(\mathbb{A})\le f(n)$ for all nonnegative primitive tensors of order $m$ and dimension $n$. \end{conjecture}

In the case of $m = 2$ ($\mathbb{A}$ is a matrix), the well-known Wielandt's upper bound tells us that we can take
$f(n) =(n-1)^2 + 1.$

Recently,  the authors \cite{YHY13} confirmed Conjecture \ref{con17} by proving Theorem \ref{thm18}.

\begin{theorem}\label{thm18}{\rm (\cite{YHY13}, Theorem 1.2)}
Let $\mathbb{A}$ be a nonnegative primitive tensor with order $m$ and dimension $n$.
Then its primitive degree $\gamma(\mathbb{A})\le (n-1)^2+1$, and the upper bound is tight.
\end{theorem}

In fact, Theorem \ref{thm18} not only confirmed the existence of $f(n)$,
but also showed that $f(n)$ is a quadratic function of $n$,
 independent of $m$,
furthermore the expression of $f(n)$ with the tensor case is the same with the matrix case.

Let  $m, n$ be positive integers with $m\ge 2, n\ge 2$, and the exponent set of  primitive tensors with order $m$ and dimension $n$ be
$$E(m,n)=\{k |\mbox { there exists a primitive tensor } \mathbb{A} \mbox { of order } m \mbox{  dimension } n
              \mbox { such that }k=\gamma(\mathbb{A}) \}.$$
            Now  it is natural to consider the question to completely determine $E(m,n)$.

 By Theorem \ref{thm18}, we know $E(m,n)\subseteq [(n-1)^2+1]$.
Clearly, $E(2,n)=E_n$,
where $E_n=\{k | \mbox { there exists a   primitive matrix }  A \mbox { of order } n \mbox { such that } $ $\gamma(A)=k\}.$

Let $A=(a_{ij})$ be a nonnegative primitive matrix of order $n$.   In 1950, H. Wielandt \cite{Wi59} first stated the sharp upper bound for
$\gamma(A)$, that is, $\gamma(A)\le w_n = (n-1)^2+1$ for all  primitive matrices  of order $n$ and thus $E_n\subseteq [1,w_n]^o$,
where  $a,b$ are positive integers with $b>a$ and $[a,b]^o=\{k | k \mbox { is an integer and } a\leq k\leq b\}$.
In 1964, A. L. Dulmage and N. S. Mendelsohn \cite{DM64}  revealed the existence of the so-called gaps in
the exponent set of  primitive matrices, that is, $E_n \subset [1,w_n]^o$,
where  ``gap" is a set  of consecutive integers $[a,b]^o(\subset [1,w_n]^o)$,  such that no   matrix $A$ of order $n$ satisfying $\gamma(A)\in [a,b]^o$.
In 1981, M. Lewin and Y. Vitek \cite{LV81}  found all gaps in $[\lfloor\frac{1}{2}w_n\rfloor+1, w_n]^o$,
and  conjectured that $[1, \lfloor\frac{1}{2}w_n\rfloor]^o$ has no gaps,
where $\lfloor x\rfloor$ denotes the greatest integer $\le x$.
 In 1985, Shao  \cite{Sh85} proved that this Lewin-Vitek Conjecture is true for all sufficiently large $n$,
and the conjecture has one counterexample when $n=11$ since $48\not\in E_{11}$.
Finally, in 1987, Zhang \cite{Zh87} continued and completed the work. He  showed that the  Lewin-Vitek Conjecture holds for all $n$
 except $n=11$. Thus the  exponent set $E_n$ for primitive matrices of order $n$ is completely determined.

For general tensors (the case of $m\geq 3$),
in \cite{HYY14}, the authors  showed that there are no gaps in tensor case when $m\ge n\ge3$.
but the tensor cases  $n>m\ge 3$ and $m>n=2$ are still open.

In this paper, we  show that there are no gaps  in tensor case $m\ge 3$   in Section 3,
that is, $E(m,n)= [(n-1)^2+1]=\{1,2,\ldots, (n-1)^2+1\}$ when $m\geq 3$,
and  propose some open problems for further  research in Section 4.

\section{Preliminaries}
\hskip.6cm  For proving Conjecture \ref{con17}, the authors \cite{YHY13} defined $j$-primitive and $j$-primitive degree for a nonnegative tensor and obtained the following result.

\begin{definition}\label{defn23}{\rm (\cite{YHY13}, Definition 2.13)}
 Let $\mathbb{A}$ be a nonnegative  tensor with order $m$ and dimension $n$. For a fixed integer $j\in[n]$, if there exists a positive integer $k$ such that
$$(M(\mathbb{A}^k))_{uj}>0, \hskip.2cm {\mbox for\,\, all } \, u \in[n],$$
then $\mathbb{A}$ is called $j$-primitive and the smallest such integer $k$ is called the $j$-primitive degree of $\mathbb{A}$, denoted by $\gamma_j(\mathbb{A})$.
\end{definition}

\begin{proposition}\label{pro24}{\rm (\cite{YHY13}, Proposition 2.14)}
Let $\mathbb{A}$ be a nonnegative primitive  tensor with order $m$ and dimension $n$. Then
$\gamma(\mathbb{A})=\max\limits_{1\le j\le n}\{\gamma_j(\mathbb{A})\}.$
\end{proposition}

Let $\mathbb{A}$ be a nonnegative tensor with order $m$ and dimension $n$. For positive integers $k$ and $j\in[n]$, a notation
$S_k(\mathbb{A}, j)$ was introduced in \cite{HYY14} as follows:
\begin{equation}\label{eq21}
S_k(\mathbb{A}, j)=\{u\in[n]\hskip.08cm |\hskip.08cm (M(\mathbb{A}^k))_{uj}>0 \}, \hskip.2cm k=1, 2, \ldots.
\end{equation}
 and by equation
\begin{equation*}
 (M(\mathbb{A}^{k+1}))_{uj}=\sum\limits_{i_2,\ldots, i_m=1}^{n}a_{ui_2\cdots i_m}(M(\mathbb{A}^k))_{i_2j}\cdots(M(\mathbb{A}^k))_{i_mj},
\end{equation*}
which investigated in \cite{YHY13}, the following recurrence relation (\ref{eq22}) was obtained as follows:

\begin{equation}\label{eq22}
S_{k+1}(\mathbb{A}, j)=\{u\in[n]\hskip.08cm|\hskip.08cm\mbox {there exist }  i_2, \ldots, i_m\in S_k(\mathbb{A}, j) \mbox{ and } a_{ui_2\cdots i_m}>0\}.
\end{equation}

According to this relation (\ref{eq22}), some good properties were obtained.
\begin{proposition}\label{pro28}{\rm (\cite{HYY14}, Lemma 3.1 and Remark 3.2)}
Let $\mathbb{A}$ be a nonnegative  tensor with order $m$ and dimension $n$.

{\rm (i) } Let $k, l, i, j$ be positive integers such that $1\le i, j\le n$. Suppose that $S_k(\mathbb{A}, i)=S_l(\mathbb{A}, j)$, then
$S_{k+r}(\mathbb{A}, i)=S_{l+r}(\mathbb{A}, j)$ holds for every positive integer $r$.

{\rm (ii) } For any $j\in[n]$, let $k$ be the least positive integer such that $S_k(\mathbb{A}, j)=[n]$.
Then for any integer $l\geq k$,  $S_l(\mathbb{A}, j)=[n]$.

{\rm (iii). }For any $j\in[n]$, $\gamma_j(\mathbb{A})$ is the least positive integer $k$ satisfying $S_k(\mathbb{A}, j)=[n]$.
\end{proposition}

\begin{remark}\label{rem211}
For convenience, we replace $S_t(\mathbb{A},n)$ by the notation $S_t(\mathbb{A},0)$, and  replace $\gamma_n(\mathbb{A})$ by $\gamma_0(\mathbb{A})$, respectively. Under these notations, we have $\gamma(\mathbb{A})=\max\limits_{0\le j\le n-1}\{\gamma_j(\mathbb{A})\}$ by Proposition \ref{pro24}.
\end{remark}

In \cite{YHY13}, the authors introduce some theoretical concepts of digraphs and matrices.

Let $D=(V,E)$ denote  a digraph on $n$ vertices with vertex set $V(D)=V$ and arc set $E(D)=E$. Loops are
permitted, but no multiple arcs. A $u\rightarrow v$ walk in
$D$ is a sequence of vertices $u, u_1,\ldots, u_k=v$ and a
sequence of arcs $e_1=(u,u_1),e_2=(u_1,u_2), \ldots,
e_k=(u_{k-1},v)$, where the vertices and the arcs are not
necessarily distinct. We use the notation
$u\rightarrow u_1\rightarrow u_2\rightarrow \cdots \rightarrow u_{k-1}\rightarrow v$
to refer to this  $u\rightarrow v$ walk.
A closed walk is a $u\rightarrow v$ walk where $u=v$. A path is a walk with distinct vertices. A
 cycle is a closed $u\rightarrow v$ walk with distinct
vertices except for $u=v$. The length of a walk $W$ is the
number of arcs in $W$, denoted by $l(W)$. 

\begin{definition}\label{defn25}{\rm (\cite{YHY13}, Definition 2.9)}
Let $D=(V,E)$ denote  a digraph on $n$ vertices.
A digraph $D^{\prime}=(V, E^{\prime})$ is called the reversed digraph of $D$
where $(j,i)\in E^{\prime}$ if and only if $(i,j)\in E$ for any $i,j\in V$, denoted by $\overleftarrow{D}$.
\end{definition}

Let $A=(a_{ij})$ be a square nonnegative matrix of order $n$.
The associated digraph $D(A)=(V, E)$ of $A$ (possibly with loops) is defined to be the digraph with vertex set
$V=\{1,2,\ldots,n\}$ and arc set $E=\{(i,j)|a_{ij}\neq 0\}$.
The associated reversed digraph $\overleftarrow{D(A)}=(V, E^{\prime})$ of $A$ (possibly with loops) is defined to be the digraph with vertex set
$V=\{1,2,\ldots,n\}$ and arc set $E^{\prime}=\{(j,i)|a_{ij}\neq 0\}$.
Clearly, the associated reversed digraph of $A$ is the reversed digraph of the associated digraph of $A$.





\begin{proposition}\label{pro29}{\rm (\cite{HYY14}, Proposition 4.1)}
Let $A$ be a nonnegative matrix of order $n$,  $j(\in [n]), k$ be positive integers,
$S_k(A,j)=\{u\in [n]  \hskip0.08cm|\hskip0.08cm (A^k)_{uj}>0\}$.
Then $$S_k(A,j)=\{u\in [n] \hskip0.08cm|\hskip0.08cm \mbox{there exists a walk of length $k$ from $j$ to $u$ in the digraph } \overleftarrow{D(A)}\}.$$
\end{proposition}

\begin{proposition}\label{pro210}{\rm (\cite{HYY14}, Lemma 4.2)}
Let $\mathbb{A}$ be a nonnegative   tensor with order $m$ and dimension $n$ such that
  $a_{ii_2\cdots i_m}=0$ if $i_2\cdots i_m\not=i_2\cdots i_2$  for any $i\in [n]$.
  Then for any positive integers $j(\in [n])$ and $k$,
  \begin{equation}\label{eq23}
  S_k(\mathbb{A}, j)=S_k(M(\mathbb{A}),j).
  \end{equation}
\end{proposition}

The relations between a tensor $\mathbb{A}$ and the majorization matrix $M(\mathbb{A})$ are important and useful.

\begin{proposition}\label{pro26}{\rm (\cite{Sh12}, Corollary 4.1 )}
   Let $\mathbb{A}$  be a nonnegative tensor with order $m$ and dimension $n$. If $M(\mathbb{A})$ is primitive, then $\mathbb{A}$ is also primitive.
   \end{proposition}

\begin{proposition}\label{pro27}{\rm (\cite{YHY13}, Corollary 3.4)}
Let $\mathbb{A}$ be a nonnegative primitive  tensor with order $m$ and dimension $n$ such that
  $a_{ii_2\cdots i_m}=0$ if $i_2\cdots i_m\not=i_2\cdots i_2$  for any $i\in [n]$.
  If $M(\mathbb{A})$ is primitive, then
   $\gamma(\mathbb{A})=\gamma(M(\mathbb{A})).$
\end{proposition}

Now we define the nonnegative tensor $\mathbb{A}_0=(a_{i_1i_2\ldots i_m})_{1\leq i_j\leq n \hskip.2cm (j=1,\ldots, m)}$ with order $m$ and dimension $n$ such that

{\rm(i). } The majorization matrix  $M(\mathbb{A}_{0})$  is given by

$$M_1=\left(
                 \begin{array}{cccccc}
                    0 & 0 & \cdots & 0 & 1 & 1 \\
                    1 & 0 & \cdots & 0 & 0 & 0 \\
                    0 & 1 & \cdots & 0 & 0 & 0\\
                    \vdots &\vdots & \ddots & \vdots & \vdots &\vdots \\
                    0 & 0 & \cdots & 1 & 0 & 0 \\
                    0 & 0 & \cdots & 0 & 1 & 0 \\
                 \end{array}
               \right).
$$


{\rm (ii). } $a_{ii_2\cdots i_m}=0$, if $i_2\cdots i_m\not=i_2\cdots i_2$ for any $i\in[n]$.

It is well-known that $M_1$ is primitive and the primitive exponent $\gamma(M_1)=(n-1)^2+1$. Then by Propositions \ref{pro26} $\sim$ \ref{pro27}, it can easily be seen that the tensor $\mathbb{A}_0$ is primitive  and its primitive degree $\gamma(\mathbb{A}_0)=(n-1)^2+1$.
Furthermore, there are more good  properties on tensor $\mathbb{A}_0$ as follows.

\begin{proposition}\label{pro212}{\rm (\cite{HYY14}, Proposition 4.5)}
Let $\mathbb{A}_0$ be the nonnegative primitive  tensor with order $m$ and dimension $n$ defined as above. Then
$\gamma_{n-1}(\mathbb{A}_0)=n^2-3n+3$.
\end{proposition}

\begin{proposition}\label{pro213}{\rm (\cite{HYY14}, Remark 4.6 and Proposition 4.9)}
Let $k$ be positive integer with $1\leq k\le n^2-3n+2$. Then
 $S_1(\mathbb{A}_0, n-1), S_2(\mathbb{A}_0, n-1), \ldots, S_k(\mathbb{A}_0, n-1)$
are pairwise distinct proper subsets of $[n]$, and $S_{n^2-3n+2}(\mathbb{A}_0, n-1)=[n-1]$.
\end{proposition}

For convenience, we  let $|a|_n$ denote the least positive integer $t$ with $t \equiv a \pmod{n}$, and
let a set $S=\{a_1,\cdots,a_s\}\pmod{n}$ denote the set $S=\{|a_1|_n,\cdots,|a_s|_n\}$.

\begin{proposition}\label{pro214}
Let $k$ be a positive integer and $1\le k\le n^2-3n+2$.

{\rm (i). }If $k=(n-1)q+r$ with $q\ge 0$ and $1\le r\le n-1$, then

\begin{equation}\label{eq24}
S_{k}(\mathbb{A}_0,n-1)=\{r-q-1,\cdots,r-1,r\}\pmod{n}
\end{equation}
and $|S_{k}(\mathbb{A}_0,n-1)|=q+2$.

{\rm (ii). } Let $t(\geq 1)$, $j(\geq 0)$ be integers. Then for any positive integer $t$ and any integer $j\in \{0,1,\ldots, n-2\}$,
we have $S_{t+(n-1-j)}(\mathbb{A}_0,j)=S_{t}(\mathbb{A}_0,n-1)$.
\end{proposition}

\begin{proof}
{\rm (i). } By the definition of $\mathbb{A}_0$ and  Propositions \ref{pro29}$\sim$\ref{pro210}, we have

$S_{k}(\mathbb{A}_0,n-1)$

 \noindent \hskip.2cm$=S_{k}(M(\mathbb{A}_0),n-1) $

 \noindent \hskip.2cm $=\{u\in [n] \hskip0.08cm|\hskip0.08cm \mbox{ there exists a walk of length }  k \mbox { from } n-1 \mbox{  to }  u
                       \mbox{ in the digraph } \overleftarrow{D(M(\mathbb{A}_0))}\}.$

Now we can obtain the conclusion (i)  from the  digraph $\overleftarrow{D(M(\mathbb{A}_0))}=\overleftarrow{D(M_1)}$ (See Figure $1$)
by simple calculation.

$$
        \hskip1cm
 \xy 0;/r3pc/: \POS (1,1) *\xycircle<3pc,3pc>{};
        \POS(1,2) *@{*}*+!D{n}="n";
        \POS(1.5,1.86) \ar@{->}(1.5,1.86);(1.6,1.8)="a";
        \POS(.5,1.86)   \ar@{->}(.5,1.86);(.6,1.91)="b";
        \POS(1.8,1.6)  *@{*}*+!L{\hspace*{3pt}{n\hspace*{-3pt}-\hspace*{-3pt}1}}="c";
        \POS(.2,1.6)   *@{*}*+!R{\mathrm{1}}="d";
        \POS(.2,.4)    \ar@{->}(.2,.4) ;(.19,.415) ="e";
        \POS(1.8,.4)   \ar@{->}(1.8,.4);(1.79,.385)="f";
        \POS "c" \ar @{->} (.7,1.6) \ar @{-} "d";
         \POS (0.015, .9) *@{*}*+!R{2}="g";
         \POS(.4,0.2) *@{*}*+!R{\mathrm{}}="k";
      \POS(.85,0.006) \ar@{->}(1,0.00);(0.85,.006)="r";
         \POS(1.6,0.2) *@{*}*+!R{\mathrm{}}="l";
           \POS(0.8,0.2) *@{*}*+!R{\mathrm{}}="m";
             \POS(1.0,0.2) *@{*}*+!R{\mathrm{}}="p";
               \POS(1.2,0.2) *@{*}*+!R{\mathrm{}}="q";

         \POS(2.0,.9)  *@{*}*+!L{\hspace*{3pt}{n\hspace*{-3pt}-\hspace*{-3pt}2}}="h";
        \POS(1.99,1.1) \ar@{->}(1.98,1.25);(1.99,1.1)="i";
        \POS(.01,1.1)   \ar@{->}(.01,1.1);(.023,1.25)="j";

 \endxy
 $$
 $$\mbox{ \qquad Figure 1.  digraph } \overleftarrow{D(M(\mathbb{A}_0))}$$

{\rm (ii). } By the definition of $\mathbb{A}_0$, an easy computation shows that $$\{n,1\}=S_1(\mathbb{A}_0,n-1)=S_2(\mathbb{A}_0,n-2)=\cdots=S_{n-1}(\mathbb{A}_0,1)=S_n(\mathbb{A}_0,0).$$

It implies that
\begin{equation}\label{eq25}
S_1(\mathbb{A}_0,n-1)=S_{n-j}(\mathbb{A}_0,j) \mbox{ for any  } j\in \{0,1,\ldots, n-2\}.
\end{equation}

\noindent Then for any integer $r\ge1$,  by (i) of Proposition \ref{pro28}, we have

\begin{equation}\label{eq26}
S_{1+r}(\mathbb{A}_0,n-1)=S_{n-j+r}(\mathbb{A}_0,j) \mbox { for any } j\in \{0,1,\ldots, n-2\}.
 \end{equation}

 Combining (\ref{eq25}) and (\ref{eq26}), (ii) holds.
\end{proof}

\begin{remark}\label{rem215}
For convenience, we regard a set $\{n,1\}$ as a consecutive positive integers subset of $\{1,2,\ldots, n\}$.
Then it implies that  $\{n-1,n, 1\}$, $\{n,1,2\}$, $\{n,1,2,3\}$ and so on, and thus $S_k(\mathbb{A}_0,n-1)$ are  consecutive positive integers subsets of $\{1,2,\cdots,n\}$
for all $k\in [n^2-3n+2]$ by the result {\rm (i) } of Proposition \ref{pro214}.
\end{remark}

\section{The exponent set of nonnegative primitive tensors}

\hskip0.6cm In this section, we will show that there are no gaps  in tensor case $m\ge 3$,
that is, $E(m,n)= [(n-1)^2+1]$ when $m\geq 3$. It implies that the results of the case $m\geq 3$ is totally different from the case $m=2$.

Let $s_1,s_2, \ldots, s_l$ be $l$ elements which are not necessarily to be  distinct.
For convenience, we let $\{s_1,s_2, \ldots, s_l\}$   denote the set of all different elements from $s_1,s_2, \ldots, s_l$.

Let $k(\geq 1), n(\geq 2), q(\geq 0), r(\geq 1)$ be  integers,
$\mathbb{A}_0=(a_{i_1i_2\ldots i_m})_{1\leq i_j\leq n \hskip.2cm (j=1,\ldots, m)}$ be a nonnegative tensor  with order $m$ and dimension $n$ defined in Section 2,  $k=(n-1)q+r$ with $1\le r\le n-1$, then $0\leq q\leq n-3$.
We define the nonnegative tensor $\mathbb{A}_k=(a^{(k)}_{i_1i_2\cdots i_m})_{1\leq i_j\leq n \hskip.2cm (j=1,\ldots, m)}$ with order $m$ and dimension $n$ such that
\vskip0.2cm
{\rm (i). }$M(\mathbb{A}_k)=M(\mathbb{A}_0)=M_1;$

{\rm (ii). }$a^{(k)}_{i\alpha}=1$, if $i\in[n]\backslash\{r-q, r-q+1, \ldots, r,r+1\} \pmod n$ and $\alpha=i_2\cdots i_m\in[n]^{m-1}$ with $\{i_2,\cdots,i_m\}=\{r-q-1,r\} \pmod n;$

{\rm (iii). }$a^{(k)}_{i_1i_2\cdots i_m}=0$, except for {\rm (i)} and {\rm (ii)}.
\vskip0.2cm
By Proposition \ref{pro26}, it's easy to see the tensor $\mathbb{A}_k$ is primitive from the fact that $M(\mathbb{A}_k)=M(\mathbb{A}_0)=M_1$ is primitive. Now we investigate the relations between the tensors $\mathbb{A}_k$ and $\mathbb{A}_0$.

\begin{proposition}\label{pro31}
Let $k(\geq 1), n(\geq 2), q(\geq 0), r$ be  integers, and $k=(n-1)q+r\in [n^2-3n+2]$  with $1\le r\le n-1$, $\mathbb{A}_0$ and $\mathbb{A}_k$ defined as above.
Then for any integer $t\in [k]$, we have
\begin{equation}\label{eq31}
 S_t(\mathbb{A}_k,n-1)=S_t(\mathbb{A}_0,n-1).
\end{equation}
\end{proposition}

\begin{proof}
Now  we prove (\ref{eq31})
holds for any $t\in [k]$  by induction on $t$.

It's obvious that $S_1(\mathbb{A}_k,n-1)=S_1(\mathbb{A}_0,n-1)=\{1,n\}$ from (\ref{eq21}) and $M(\mathbb{A}_k)=M(\mathbb{A}_0)$. Assume that $S_t(\mathbb{A}_k,n-1)=S_t(\mathbb{A}_0,n-1)$ holds for $t<k$.
Now we only need  show that $S_{t+1}(\mathbb{A}_k, n-1)=S_{t+1}(\mathbb{A}_0,n-1)$ holds. 

{\bf Case 1: } $\{r-q-1,r\} \pmod n\not\subseteq S_t(\mathbb{A}_0,n-1)$.

 We note that the recurrence relation (\ref{eq22}) and the definition of $\mathbb{A}_k$, then

\vskip0.2cm

$S_{t+1}(\mathbb{A}_k,n-1)=\{u\in[n] \hskip.08cm|\hskip.08cm \mbox {there exist }  i_2, \ldots, i_m\in S_t(\mathbb{A}_k, n-1) \mbox{ and } a^{(k)}_{ui_2\cdots i_m}>0\}$

\hskip2.85cm$=\{u\in[n]\hskip.08cm |\hskip.08cm \mbox {there exist } i_2, \ldots, i_m\in S_t(\mathbb{A}_{0}, n-1) \mbox{ and } a^{(k)}_{ui_2\cdots i_m}>0\}$

\hskip2.85cm$=\{u\in[n]\hskip.08cm |\hskip.08cm \mbox {there exist }  v\in S_t(\mathbb{A}_{0}, n-1) \mbox{ and } a_{uv\cdots v}>0\}$

\hskip2.9cm$\cup\hskip.08cm\{u\in[n] \hskip.08cm|\hskip.08cm \mbox {there exist }  i_2, \ldots, i_m\in S_t(\mathbb{A}_0, n-1) \mbox{ such that } $

\hskip3.5cm$i_2\cdots i_m\not=i_2\cdots i_2 \mbox{ and } a^{(k)}_{ui_2\cdots i_m}>0\}$

\hskip2.85cm$=S_{t+1}(\mathbb{A}_0,n-1)\hskip.08cm\cup  \varnothing$

\hskip2.85cm$ =S_{t+1}(\mathbb{A}_0,n-1).$
\vskip0.2cm

{\bf Case 2: } $\{r-q-1,r\} \pmod n\subseteq S_t(\mathbb{A}_0,n-1)$.

From the Remark \ref{rem215}, we know that $S_t(\mathbb{A}_0,n-1)$ is a consecutive positive integers subset of $\{1,2,\cdots,n\}$. Thus $$\{r-q-1,r-q,\cdots,r\} \pmod n\subseteq S_t(\mathbb{A}_0,n-1)$$
or
$$\{r,r+1,\cdots,r-q-1\} \pmod n\subseteq S_t(\mathbb{A}_0,n-1)$$
holds.


If $\{r-q-1,r-q,\cdots,r\} \pmod n \subseteq S_t(\mathbb{A}_0,n-1)$,
then $S_k(\mathbb{A}_0,n-1)\subseteq S_t(\mathbb{A}_0,n-1)$
and $|S_t(\mathbb{A}_0,n-1)|=\lfloor\frac{t-1}{n-1}\rfloor+2\leq |S_k(\mathbb{A}_0,n-1)|=q+2$
by the result (i) of Proposition \ref{pro214} and $t<k$. Then  $S_t(\mathbb{A}_0,n-1)=S_k(\mathbb{A}_0,n-1)$,
it implies a contradiction since $S_t(\mathbb{A}_0,n-1)\not=S_k(\mathbb{A}_0,n-1)$ by $t<k$ and Proposition \ref{pro213}.
It follows that $$\{r,r+1,\cdots,r-q-1\}\pmod n\subseteq S_t(\mathbb{A}_0,n-1).$$
Thus by the definition of $\mathbb{A}_0$ and  Propositions \ref{pro29} $\sim$ \ref{pro210}, we have
\begin{equation}\label{eq32}
\{r+1,r+2,\cdots, r-q\}\pmod n\subseteq S_{t+1}(\mathbb{A}_0,n-1).
\end{equation}

Then by the definition of $\mathbb{A}_k$, the assumption and (\ref{eq32}), we can see that
\vskip0.2cm

$S_{t+1}(\mathbb{A}_k,n-1)=\{u\in[n] \hskip.08cm|\hskip.08cm \mbox {there exist }  i_2, \ldots, i_m\in S_t(\mathbb{A}_k, n-1) \mbox{ and } a^{(k)}_{ui_2\cdots i_m}>0\}$

\hskip2.85cm$=\{u\in[n]\hskip.08cm |\hskip.08cm \mbox {there exist }  i_2, \ldots, i_m\in S_t(\mathbb{A}_{0}, n-1) \mbox{ and } a^{(k)}_{ui_2\cdots i_m}>0\}$

\hskip2.85cm$=\{u\in[n]\hskip.08cm |\hskip.08cm \mbox {there exist }  v\in S_t(\mathbb{A}_{0}, n-1) \mbox{ and } a_{uv\cdots v}>0\}$

\hskip2.9cm$\cup\hskip.08cm\{u\in[n] \hskip.08cm|\hskip.08cm \mbox {there exist }  i_2, \ldots, i_m\in S_t(\mathbb{A}_0, n-1) \mbox{ such that } $

\hskip3.5cm$i_2\cdots i_m\not=i_2\cdots i_2 \mbox{ and } a^{(k)}_{ui_2\cdots i_m}>0\}$

\hskip2.85cm$=S_{t+1}(\mathbb{A}_0,n-1)\hskip.08cm\cup \hskip.08cm([n]\backslash\{r-q, r-q+1, \ldots, r,r+1\} \pmod n)$

\hskip2.85cm$ =S_{t+1}(\mathbb{A}_0,n-1)$.
\vskip0.2cm

Combining the above two cases,   the equation (\ref{eq31}) is proved for all $t\in [k]$.
\end{proof}

\begin{proposition}\label{pro32}
Let $k(\geq 1), n(\geq 2), q(\geq 0), r$ be  integers, and $k=(n-1)q+r\in [n^2-3n+2]$  with $1\le r\le n-1$,
$\mathbb{A}_0$ and $\mathbb{A}_k$ defined as above.
Then for any integer $t\in [k]$, 
we have
\begin{equation}\label{eq33}
 S_t(\mathbb{A}_k,j)=S_t(\mathbb{A}_0,j),  \mbox{where } j\in \{0,1,2,\ldots, n-2\}
\end{equation}
\end{proposition}
\begin{proof}
Now we show (\ref{eq33})
holds for any  $t\in [k]$ by the following two cases.

 {\bf Case 1: } $1\le t\le n-1$.

 Now  we prove (\ref{eq33}) holds for  any  $t \hskip.1cm (1\leq t\leq n-1)$  by induction on $t$.

Clearly, $S_1(\mathbb{A}_k,j)=\{j+1\}=S_1(\mathbb{A}_0,j)$ for any  $j\in \{0,1,2,\ldots, n-2\}$.
Now we assume that $S_t(\mathbb{A}_k,j)=S_t(\mathbb{A}_0,j)$ holds for any  $j\in \{0,1,2,\ldots, n-2\}$ and $t<n-1$.
Then we show that $S_{t+1}(\mathbb{A}_k,j)=S_{t+1}(\mathbb{A}_0,j)$ holds for any  $j\in \{0,1,2,\ldots, n-2\}$.

By observing and the result  (ii) of Proposition \ref{pro214}, we know for any $t$ with  $2\le t\leq  n-1$,

$S_t(\mathbb{A}_0,j)$

\noindent $=\left\{\begin{array}{ll}
                               \{j+t\}, & j\in\{0,1,\ldots, n-t-1\}; \\
                               S_{t-(n-1-j)}(\mathbb{A}_0,n-1)=\{t+j-n, t+j-n+1\}\pmod n, & j\in\{n-t, \ldots, n-2\}.
                             \end{array}
                             \right.$

We note that $(t+j-n+1)-(t+j-n)\equiv 1\pmod n$,
but $(r-q-1)-r\not\equiv 1\pmod n$  by $0\leq q\leq n-3$.
Thus  $\{r-q-1,r\}\pmod n\subseteq S_t(\mathbb{A}_0,j)$  if and only if $q=0$ and $r=k$,
then $\{r-q-1,r\}\pmod n=\{k-1,k\}\pmod n$. It is a contradiction by the fact $1\leq t+j-n+1<t\leq k$ for $j\in\{n-t, \ldots, n-2\}$.

The above arguments imply  that $\{r-q-1,r\}\pmod n\not\subseteq S_t(\mathbb{A}_0,j)$ for $j=0,1,\cdots,n-2$.
Then by the definition of $\mathbb{A}_k$, and the assumption, we can see that

$S_{t+1}(\mathbb{A}_k,j)=\{u\in[n] \hskip.08cm|\hskip.08cm \mbox {there exist }  i_2, \ldots, i_m\in S_t(\mathbb{A}_k, j) \mbox{ and } a^{(k)}_{ui_2\cdots i_m}>0\}$

\hskip2.05cm$=\{u\in[n]\hskip.08cm |\hskip.08cm \mbox {there exist }  i_2, \ldots, i_m\in S_t(\mathbb{A}_{0}, j) \mbox{ and } a^{(k)}_{ui_2\cdots i_m}>0\}$

\hskip2.05cm$=\{u\in[n]\hskip.08cm |\hskip.08cm \mbox {there exist }  v\in S_t(\mathbb{A}_{0}, j) \mbox{ and } a_{uv\cdots v}>0\}$

\hskip2.1cm$\cup\hskip.08cm\{u\in[n] \hskip.08cm|\hskip.08cm \mbox {there exist }  i_2, \ldots, i_m\in S_t(\mathbb{A}_0, j) \mbox{ such that } $

\hskip2.25cm$i_2\cdots i_m\not=i_2\cdots i_2 \mbox{ and } a^{(k)}_{ui_2\cdots i_m}>0\}$

\hskip2.05cm$=S_{t+1}(\mathbb{A}_0,j)\hskip.08cm\cup \varnothing$

\hskip2.05cm$ =S_{t+1}(\mathbb{A}_0,j)$.
\vskip0.2cm

{\bf Case 2: } $t\geq n$.

Now  we prove (\ref{eq33}) holds for  any  $n\leq t\leq k$  by induction on $t$.

By (ii) of Proposition \ref{pro214}, $S_n(\mathbb{A}_0,j)=S_{j+1}(\mathbb{A}_0,n-1)=\{j,j+1\}\pmod n$ for  any $j\in \{0,1,\ldots, n-2\}$.
On the other hand, by Proposition \ref{pro210}, the result of  Case 1 and (\ref{eq22}), we have $$S_{n-1}(\mathbb{A}_{0}, j)=\left\{\begin{array}{ll}
                                                                                            \{n-1\}, & j=0; \\
                                                                                            \{j-1,j\}\pmod n, & 1\leq j\leq n-2.
                                                                                          \end{array}
                                                                                          \right.$$
and

$S_{n}(\mathbb{A}_k,j)=\{u\in[n] \hskip.08cm|\hskip.08cm \mbox {there exist }  i_2, \ldots, i_m\in S_{n-1}(\mathbb{A}_k, j) \mbox{ and } a^{(k)}_{ui_2\cdots i_m}>0\}$

\hskip1.75cm$=\{u\in[n]\hskip.08cm |\hskip.08cm \mbox {there exist }  i_2, \ldots, i_m\in S_{n-1}(\mathbb{A}_{0}, j) \mbox{ and } a^{(k)}_{ui_2\cdots i_m}>0\}$

\hskip1.75cm$=\left\{\begin{array}{ll}
                                                                                            \{1,n\}, & j=0; \\
                                                                                            \{j,j+1\}\pmod n, & 1\leq j\leq n-2.
                                                                                          \end{array}
                                                                                          \right.$

\hskip1.75cm$=\{j,j+1\}\pmod n.$

Thus  $S_{n}(\mathbb{A}_k,j) =S_{n}(\mathbb{A}_0,j)$.
\vskip0.2cm

Now we assume that $S_t(\mathbb{A}_k,j)=S_t(\mathbb{A}_0,j)$ holds for $n\le t<k$ for any $j\in \{0,1,\ldots, n-2\}$.
Then we only need show that $S_{t+1}(\mathbb{A}_k,j)=S_{t+1}(\mathbb{A}_0,j)$ holds for any  $j\in \{0,1,\ldots, n-2\}$.
By the assumption, the fact that $1\leq t-(n-1-j)<k$, (ii) of  Proposition \ref{pro214} and  Proposition \ref{pro31}, we have

$$S_t(\mathbb{A}_k,j)=S_t(\mathbb{A}_0,j)=S_{t-(n-1-j)}(\mathbb{A}_0,n-1)=S_{t-(n-1-j)}(\mathbb{A}_k,n-1).$$

Thus  by (i) of Proposition \ref{pro28}, Proposition \ref{pro31} and (ii) of Proposition \ref{pro214}, we have
\begin{align*}
S_{t+1}(\mathbb{A}_k,j)&=S_{t+1-(n-1-j)}(\mathbb{A}_k,n-1) \\
                       &=S_{t+1-(n-1-j)}(\mathbb{A}_0,n-1) \\
                       &=S_{t+1}(\mathbb{A}_0,j) .
\end{align*}

Combining the above two cases, we complete the proof.
\end{proof}

\begin{theorem}\label{thm33}
Let $k$ be a positive integer with $1\le k\le n^2-3n+2$. Then

{\rm (i). } $\gamma_j({\mathbb{A}_k})=k+n-j$ where  $j=0,1,\cdots, n-1.$

{\rm (ii). } $\gamma(\mathbb{A}_k)=k+n.$
\end{theorem}

\begin{proof}Firstly, we show  $\gamma_{n-1}({\mathbb{A}_k})=k+1.$
By Proposition \ref{pro31} and (i) of  Proposition \ref{pro214}, we have
$$S_{k}(\mathbb{A}_k,n-1)=S_{k}(\mathbb{A}_0,n-1)=\{r-q-1,r-q,\cdots,r\}\pmod n.$$

 \noindent And by Proposition \ref{pro210} and  (i) of  Proposition \ref{pro214}, we have
  $$\{r-q,r-q+1,\cdots,r+1\}\pmod n\subseteq S_{k+1}(M(\mathbb{A}_0),n-1)=S_{k+1}(\mathbb{A}_0,n-1).$$

\noindent   Thus by Proposition \ref{pro31}, we have

$S_{k+1}(\mathbb{A}_k,n-1)=\{u\in[n] \hskip.08cm|\hskip.08cm \mbox {there exist }  i_2, \ldots, i_m\in S_k(\mathbb{A}_k, n-1) \mbox{ and } a^{(k)}_{ui_2\cdots i_m}>0\}$

\hskip2.85cm$=\{u\in[n]\hskip.08cm |\hskip.08cm \mbox {there exist }  i_2, \ldots, i_m\in S_k(\mathbb{A}_{0}, n-1) \mbox{ and } a^{(k)}_{ui_2\cdots i_m}>0\}$

\hskip2.85cm$=\{u\in[n]\hskip.08cm |\hskip.08cm \mbox {there exist }  v\in S_k(\mathbb{A}_{0}, n-1) \mbox{ and } a_{uv\cdots v}>0\}$

\hskip2.90cm$\cup\hskip.08cm\{u\in[n] \hskip.08cm|\hskip.08cm \mbox {there exist }  i_2, \ldots, i_m\in S_k(\mathbb{A}_0, n-1) \mbox{ such that } $

\hskip3.5cm$i_2\cdots i_m\not=i_2\cdots i_2 \mbox{ and } a^{(k)}_{ui_2\cdots i_m}>0\}$

\hskip2.85cm$=S_{k+1}(\mathbb{A}_0,n-1)\hskip.08cm\cup \hskip.08cm ([n]\backslash\{r-q, r-q+1, \ldots, r,r+1\} \pmod n) $

\hskip2.85cm$=[n]$


On the other hand,  we have $S_t(\mathbb{A}_k,n-1)=S_t(\mathbb{A}_0,n-1)\subsetneqq[n]$ holds for all $1\leq t\leq k$
by Proposition \ref{pro31} and  Proposition \ref{pro213}.
Thus $\gamma_{n-1}(\mathbb{A}_k)=k+1$ by (iii) of Proposition \ref{pro28}.

Now we show  $\gamma_j({\mathbb{A}_k})=k+n-j$  for any $j\in\{0,1,\ldots, n-2\}$.

By Propositions \ref{pro31}$\sim$ \ref{pro32}, we have 
$$S_1(\mathbb{A}_k, n-1)=S_2(\mathbb{A}_k, n-2)=S_3(\mathbb{A}_k, n-3)=\cdots =S_{n-1}(\mathbb{A}_k, 1)=S_n(\mathbb{A}_k, n)=\{1,n\}.$$
Thus by (i) of Proposition \ref{pro28} and the definition of $j$-primitive degree, we have

$S_1(\mathbb{A}_k, n-1)=S_2(\mathbb{A}_k, n-2)$

\noindent $\Longrightarrow
\left\{\begin{array}{ll}
   ([n]\not=)S_{1+r}(\mathbb{A}_k, n-1)=S_{2+r}(\mathbb{A}_k, n-2), & \mbox{ if } 1\le r\le k-1; \\
    ([n]=)S_{1+k}(\mathbb{A}_k, n-1)=S_{2+k}(\mathbb{A}_k, n-2), &     \mbox{ if } r=k.
      \end{array}\right.$

\noindent $\Longrightarrow \gamma_{n-2}(\mathbb{A}_k)=k+2$.

Similarly, we can prove $\gamma_{n-3}(\mathbb{A}_k)=k+3,  \ldots,  \gamma_{1}(\mathbb{A}_k)=k+n-1, \gamma_{0}(\mathbb{A}_k)=k+n.$

Combining the above arguments, we obtain  $\gamma_j({\mathbb{A}_k})=k+n-j$ for any  $j\in \{0,1,\ldots, n-1\}.$

By Proposition \ref{pro24} or Remark \ref{rem211},   we have
$$\gamma(\mathbb{A}_k)=\max\limits_{0\le j\le n-1}\{\gamma_j(\mathbb{A}_k)\}=\gamma_0(\mathbb{A}_k)=k+n,$$
and we complete the proof of (ii) immediately.
\end{proof}

\begin{theorem}\label{thm34}
Let  $m, n$ be positive integers with $m\ge 3, n\ge 2$. Then $E(m,\,n)=[(n-1)^2+1]$.
\end{theorem}
\begin{proof} Let $t$ be any positive integer with $1\le t\le (n-1)^2+1$, we will complete the proof by constructing primitive tensor $\mathbb{B}_t$ with order $m$ and dimension $n$ such that $\gamma(\mathbb{B}_t)=t$. We consider the following two cases.

{\bf Case 1: }  $1\le t\le n$.

It is well known that there exists a primitive matrix $A_t$ of order $n$ such that $\gamma(A_t)=t$.
We define the tensor $\mathbb{B}_t$
to be  the nonnegative primitive  tensor with order $m$ and dimension $n$ such that
  $(\mathbb{B}_t)_{ii_2\cdots i_m}=0$ if $i_2\cdots i_m\not=i_2\cdots i_2$  for any $i\in [n]$, and $M(\mathbb{B}_t)=A_t$.
  Then $\gamma(\mathbb{B}_t)=\gamma(A_t)=t$ by Proposition \ref{pro27}.

 {\bf Case 2: } $n+1\le t\le (n-1)^2+1$.

 We choose $\mathbb{B}_t=\mathbb{A}_{t-n}$. Then $\gamma(\mathbb{B}_t)=(t-n)+n=t$ by Theorem \ref{thm33}. 
 \end{proof}

\section{Some open problems for further research}
\hskip.6cm In this section, we  will  propose some interesting  open problems for further  research.

Let $\mathbb{A}$ be a nonnegative primitive tensor  with order $m$ and dimension $n$. By Theorem \ref{thm18},
we have $\gamma(\mathbb{A})\leq (n-1)^2+1$. It is well-known that when $m=2$,
$\gamma(\mathbb{A})= (n-1)^2+1$ if and only if $\mathbb{A}\cong M_1$.
But the case of  $m\geq 3$ is totally different. In fact, we know $\mathbb{A}_{n^2-3n+2}\not\cong \mathbb{A}_{0}$,
but $\gamma(\mathbb{A}_{n^2-3n+2})=\gamma(\mathbb{A}_{0})=(n-1)^2+1$ by Proposition \ref{pro27} and Theorem \ref{thm33}.

 \begin{problem}\label{pro41}
Let $\mathbb{A}$ be a nonnegative  primitive tensor  with order $m$ and dimension $n$.
If $\gamma(\mathbb{A})=(n-1)^2+1$,  can we give a characterization of such $\mathbb{A}$?
\end{problem}

Let $k$ be a positive integer with $1\leq k\leq n^2-3n+2$, $\mathbb{A}_k$ defined as Section 3.
Then by Theorem \ref{thm33}, we have
\begin{equation}\label{eq41}
\{\gamma_1(\mathbb{A}_k), \gamma_2(\mathbb{A}_k), \ldots, \gamma_{n-1}(\mathbb{A}_k),\gamma_n(\mathbb{A}_k)\}
=\{k+1, k+2, \ldots, k+n\}=[k+1,k+n]^o.
\end{equation}
Based on (\ref{eq41}), we propose the following question.

\begin{problem}\label{pro42}
Let $\mathbb{A}$ be a nonnegative primitive tensor  with order $m$ and dimension $n$.
Does there exist two positive integers $a,b$ with $a<b$ such that
$\{\gamma_1(\mathbb{A}), \gamma_2(\mathbb{A}), \ldots, \gamma_{n-1}(\mathbb{A}),\gamma_n(\mathbb{A})\}$
$=[a,b]^o?$
 \end{problem}

Now we recall the definition of reducibility.


\begin{definition}{\rm (\cite{Ch1}, Definition 2.1)}\label{defn43}
 A tensor $\mathbb{C}=(c_{i_1\ldots i_m})$ of order $m$ dimension $n$ is
called reducible, if there exists a nonempty proper index subset $I\subset \{1, \ldots, n\}$
such that $$c_{i_1\ldots i_m}=0 \hskip.3cm \forall i_1\in I, \forall i_2,\ldots i_m\not\in I.$$
If $\mathbb{C}$ is not reducible, then we call $\mathbb{C}$ irreducible.
\end{definition}

When $m=2$, Definition \ref{defn43} give the definition of reducible matrices and irreducible matrices.
It is well-known from the basic relations between matrices and digraphs that
a square matrix $A$ is irreducible if and only if its associated digraph $D(A)$ is strongly connected,
a matrix $A$ is primitive if and only if $A$ is irreducible and the period of $A$ is 1.
It is also well-known that $A$ is primitive if and only if $D(A)$ is primitive,
  and  a digraph $D$ is primitive if and only if $D$ is strongly connected and the greatest  common divisor of
the lengths of all the cycles of $D$ is 1.   Then by the definitions of primitive matrices,
irreducible matrices and $j$-primitive, we obtain the following characterization.

\begin{proposition}\label{pro44}
Let $A$ be a nonnegative matrix of order $n$. Then $A$ is primitive if and only if $A$ is irreducible and there exists
some $j\in [n]$ such that $A$ is $j$-primitive.
\end{proposition}
\begin{proof}
Necessity is obvious. So we omit it. Now we only show sufficiency.
Let $\gamma_j(A)=k$, it implies that $(A^k)_{uj}>0$ for any $u\in [n]$, thus there exists a walk of length $k$ from $u$ to $j$ in the associated digraph $D(A)$ by the relation between matrices and digraphs.
Now we show there exists a positive  integer $l$ such that $(A^l)_{uv}>0$  for any $u,v\in [n]$,
that is, we need show there exists a walk of length $l$ from $u$ to $v$ in the associated digraph $D(A)$ for any $u,v\in V(D(A))$.

Since $A$ is irreducible, $D(A)$ is strong connected. Then for any $v$,
there exists a path $P$ with length $l_1(0\leq l_1\leq n-1)$ from $j$ to $v$.
 Similarly, there exists $w$ such that there exists a walk $Q$ with length $n-1-l_1$ from $u$ to $w$.
 Since $\gamma_j(A)=k$, then there exists a walk $R$ with length $k$ from $w$ to $j$.
 Thus the  walk $Q+R+P$  is a walk of  length $(n-1-l_1)+k+l_1=k+n-1$ from $u$ to $v$.
 Let $l=k+n-1$, we complete the proof.
\end{proof}

Based on Proposition \ref{pro44}, we propose the following conjecture.

 \begin{conjecture}\label{con45}
Let $\mathbb{A}$ be a nonnegative  tensor  with order $m$ and dimension $n$.
Then $\mathbb{A}$ is primitive if and only if $\mathbb{A}$ is irreducible and there exists
some $j\in [n]$ such that $\mathbb{A}$ is $j$-primitive.
\end{conjecture}

 Let $m(\geq2) ,n(\geq 2), j\in [n], k(\geq1)$ be positive integers,
 $\mathbb{A}$ be a nonnegative tensor  with order $m$ and dimension $n$.
  Recall the definitions of $\gamma_j(\mathbb{A})$, $S_k(\mathbb{A},j)$
  and Proposition \ref{pro28}, we know $\gamma_j(\mathbb{A})$ is the least positive integer satisfying $S_k(\mathbb{A},j)=[n]$.
  If $\mathbb{A}$ is primitive, then $1\leq \gamma_j(\mathbb{A})\leq (n-1)^2+1$ and the upper and lower bounds are tight
  by Theorem \ref{thm18} and Proposition \ref{pro24}.
   But if $\mathbb{A}$ is not  primitive, and there exists some $j\in [n]$ and positive integer $k$  such that $S_k(\mathbb{A},j)=[n]$,
    that is, there exists some $j\in [n]$  such that $\mathbb{A}$ is $j$-primitive,
    does $\gamma_j(\mathbb{A})$ have a tight lower bound and a tight upper bound?
    It is obvious that $\gamma_j(\mathbb{A})\geq 1$. Now we only consider the upper bound of $\gamma_j(\mathbb{A})$.

Let  $m, n, j(\in [n])$ be positive integers with $m\ge 2, n\ge 2$,
  and  notations 

\noindent $r(m,n)=\max\{\gamma_j(\mathbb{A})\hskip.1cm | \hskip.1cm \mathbb{A} \mbox { is a nonnegative but not primitive tensor  of order } m \mbox{  dimension } n$

\hskip3.4cm $ \mbox { and there exists some } j\in [n] \mbox{  such that } \mathbb{A} \mbox{ is } j\mbox{-primitive}\}.$

\noindent $R_j(m,n)=\{k\hskip.1cm |\hskip.1cm  \mbox {there exists  a nonnegative but not primitive tensor } \mathbb{A}
  \mbox{ of order } m \mbox{  dimension } $

\hskip2.3cm $ n \mbox { such that } \mathbb{A} \mbox{ is } j\mbox{-primitive and  } \gamma_j(\mathbb{A})=k\}.$

   Now we investigate the properties of $r(m,n)$ and $R_j(m,n)$.
\begin{proposition}\label{pro46}
$r(2,n)=n^2-4n+6.$
\end{proposition}
\begin{proof}
Let $A$ be a  nonnegative matrix of order $n$, but $A$ is not primitive  and there exists
some $j\in [n]$ such that $A$ is $j$-primitive. Let $\gamma_j(A)=k$, it implies that $(A^k)_{uj}>0$ for any $u\in [n]$,
thus there exists a walk of length $k$ from $u$ to $j$ in the associated digraph $D=D(A)$ for any $u\in [n]$ by the relation between matrices and digraphs.

Take $u=j$, we know there exists a closed walk of length $k$ from $j$ to $j$,
it implies that $j$ belongs in some strong connected subdigraph of $D$.
Without loss of generality, we assume  a set $V_1$ is the maximal subset of $V(D)$ such that  $j\in V_1\subsetneqq V(D)$
and the induced subdigraph $D_1=D[V_1]$  is strong connected.
Thus for any $v\not\in V_1$, there does not exist a walk from $j$ to $v$ in $D$ since there exists a walk of length $k$ from $v$ to $j$ by $A$ is $j$-primitive.

Now we show the adjacent matrix $A(D_1)$ is $j$-primitive.
That is, for any $u\in V_1$, let $W$ be a walk of length $k$ from $u$ to $j$ in $D$,
we only need show  $V(W)\subseteq V_1$.
Otherwise, if there  exists vertex $v\in (V(D)\backslash V_1)\cap V(W)$,
then there exists a walk from $j$ to $v$ since $u,j\in V_1$ and  $D_1=D[V_1]$  is strong connected,
and thus $v\in V_1$, it is a contradiction. 

Therefore $A(D_1)$ is primitive by Proposition \ref{pro44}
and thus there exists a positive integer $l\leq (|V_1|-1)^2+1$
such that there exists a walk of length $l$ from $u$ to $j$ for any $u\in V_1$ by the well-known Wieland's upper bound.

Then $k\leq n-|V_1|+l\leq |V_1|^2-3|V_1|+n+2\leq n^2-4n+6$ by $|V_1|\leq n-1$, and thus $r(2,n)\leq n^2-4n+6$ by the definition of $r(2,n)$.

On the other hand, let $M_2$ be a  nonnegative matrix of order $n$ and $D(M_2)$ be the associated digragh as follows.
Clearly, $M_2$ and $D(M_2)$ are not primitive.
$$M_2=\left(
                 \begin{array}{cccccccc}
                    0 & 1 & 0  &\cdots & 0 & 0 & 0\\
                    0 & 0 & 1  &\cdots & 0 & 0 & 0\\
                    \vdots & \vdots & \vdots & \ddots & \vdots  &\vdots&\vdots \\
                    0 & 0 & 0  &\cdots  & 1 &  0 & 0 \\
                    1 & 0 & 0  &\cdots  & 0 & 1 &0 \\
                    1 & 0 & 0  &\cdots  & 0 & 0 & 0 \\
                    0 & 0 & 0  &\cdots  & 0 & 1 & 0
                 \end{array}
               \right),
$$

$$
        \hskip1cm
 \xy 0;/r3pc/: \POS (1,1) *\xycircle<3pc,3pc>{};
 \POS(2.5,2.5)   *@{*}*+!D{n}="t";
 \POS(1.9,2.3) \ar@{->}(1,2);(2.5,2.5)="s";
        \POS(1,2) *@{*}*+!L{\hspace*{9pt}{n\hspace*{-3pt}-\hspace*{-3pt}1}}="r";
        \POS(1.5,1.86) \ar@{->}(1.5,1.86);(1.6,1.8)="a";
        \POS(.5,1.86)   \ar@{->}(.5,1.86);(.6,1.91)="b";
        \POS(1.8,1.6)  *@{*}*+!L{\hspace*{3pt}{n\hspace*{-3pt}-\hspace*{-3pt}2}}="c";
        \POS(.2,1.6)   *@{*}*+!R{\mathrm{1}}="d";
        \POS(.2,.4)    \ar@{->}(.2,.4) ;(.19,.415) ="e";
        \POS(1.8,.4)   \ar@{->}(1.8,.4);(1.79,.385)="f";
        \POS "c" \ar @{->} (.7,1.6) \ar @{-} "d";
         \POS (0.015, .9) *@{*}*+!R{2}="g";
         \POS(.4,0.2) *@{*}*+!R{\mathrm{}}="k";
      \POS(.85,0.006) \ar@{->}(1,0.00);(0.85,.006)="r";
         \POS(1.6,0.2) *@{*}*+!R{\mathrm{}}="l";
           \POS(0.8,0.2) *@{*}*+!R{\mathrm{}}="m";
             \POS(1.0,0.2) *@{*}*+!R{\mathrm{}}="p";
               \POS(1.2,0.2) *@{*}*+!R{\mathrm{}}="q";

         \POS(2.0,.9)  *@{*}*+!L{\hspace*{3pt}{n\hspace*{-3pt}-\hspace*{-3pt}3}}="h";
        \POS(1.99,1.1) \ar@{->}(1.98,1.25);(1.99,1.1)="i";
        \POS(.01,1.1)   \ar@{->}(.01,1.1);(.023,1.25)="j";

 \endxy
 $$
 $$\mbox{ \qquad Figure 2.  digraph } D(M_2)$$

 It is easy to show there does not exist a walk of length $n^2-4n+5$ from $n$ to $n-1$. Then $\gamma_{n-1}(M_2)\geq n^2-4n+6.$

 Combining the above two inequalities, we obtain $r(2,n)= n^2-4n+6$.
\end{proof}

The following necessary condition for a tensor to be primitive is useful.
\begin{proposition}\label{pro47}{\rm (\cite{YHY13}, (i) of Proposition 2.7)}
Let $\mathbb{A}$ be a nonnegative primitive
tensor with order $m$ and dimension $n$, $M(\mathbb{A})$ the majorization matrix of $\mathbb{A}$. Then 
for each $j\in[n]$, there exists an integer $i\in[n]\backslash\{j\}$ such that $(M(\mathbb{A}))_{ij}>0$.
\end{proposition}

\begin{proposition}\label{pro48}
If $m\geq \lfloor\frac{n-1}{2}\rfloor+1$, then $\left(\begin{array}{c}
                                                   n-1 \\
                                                   \lfloor\frac{n-1}{2}\rfloor
                                                 \end{array}\right)+1\leq r(m,n)\leq 2^n-1.$
\end{proposition}
\begin{proof}
Let  $S_1, S_2, \ldots, S_k$ be  all the subset of $[n] \backslash \{1\}$ with $|S_1|=|S_2|=\ldots=|S_k|=\lfloor\frac{n-1}{2}\rfloor$.
It is clear that $k=\left(\begin{array}{c}
                                                   n-1 \\
                                                   \lfloor\frac{n-1}{2}\rfloor
                                                 \end{array}\right)$.
Let $\mathbb{A}=(a_{i_1\ldots i_m})$ be a nonnegative tensor with  order $m$ and  dimension $n$,
where 
$$a_{i_1i_2\ldots i_m}=\left\{\begin{array}{ll}
                            1, & \mbox{ if } i_1\in S_1, \mbox{ and } i_2=\ldots=i_m=1; \\
                            1, & \mbox{ if } i_1\in S_{j+1},  \mbox{ and } \{i_2, \ldots, i_m\}=S_j \mbox{ for } j=1,2,\ldots, k-1;\\
                            1, & \mbox{ if } i_1\in [n],  \mbox{ and } \{i_2, \ldots, i_m\}=S_k; \\
                            0, & \mbox { otherwise.  }
                          \end{array}\right.$$

By the above definition, we know for any $j\in [n]\backslash \{1\}$, $(M(\mathbb{A}))_{ij}=0$ for any $i\in [n]$.
Then $\mathbb{A}$ is not primitive by Proposition \ref{pro47}.

 Now we show $\gamma_{1}(\mathbb{A})=k+1$. By (\ref{eq21}) and (\ref{eq22}), we have
$$S_1(\mathbb{A}, 1)=\{u\in [n]\hskip.1cm |\hskip.1cm M(\mathbb{A})_{u1}=a_{u1\ldots 1}>0\}=S_1,$$
$$S_2(\mathbb{A}, 1)=\{u\in [n]\hskip.1cm |\hskip.1cm \mbox{there exist } i_2,\ldots, i_m\in S_1(\mathbb{A},1) \mbox{ and } a_{ui_2\ldots i_m}>0\}=S_2,$$
$$\cdots,$$
$$S_l(\mathbb{A}, 1)=\{u\in [n]\hskip.1cm |\hskip.1cm \mbox{there exist } i_2,\ldots, i_m\in S_{l-1}(\mathbb{A},1) \mbox{ and } a_{ui_2\ldots i_m}>0\}=S_l, $$
$$\cdots,$$
$$S_k(\mathbb{A}, 1)=\{u\in [n]\hskip.1cm |\hskip.1cm \mbox{there exist } i_2,\ldots, i_m\in S_{k-1}(\mathbb{A},1) \mbox{ and } a_{ui_2\ldots i_m}>0\}=S_k,$$
$$S_{k+1}(\mathbb{A}, 1)=\{u\in [n]\hskip.1cm |\hskip.1cm \mbox{there exist } i_2,\ldots, i_m\in S_{k}(\mathbb{A},1) \mbox{ and } a_{ui_2\ldots i_m}>0\}=[n].$$
Then $\gamma_{1}(\mathbb{A})=k+1$  by Proposition \ref{pro28} and thus $r(m,n)\geq \left(\begin{array}{c}
                                                   n-1 \\
                                                   \lfloor\frac{n-1}{2}\rfloor
                                                 \end{array}\right)+1$. 

Next, we show $r(m,n)\leq 2^n-1.$
Let $\mathbb{B}=(b_{i_1\ldots i_m})$ be a nonnegative tensor with  order $m$ and  dimension $n$,
and $\mathbb{B}$ be not primitive. If $\mathbb{B}$ is $j$-primitive with some $j\in [n]$,
then we denote $K=\gamma_{j}(\mathbb{A})$. Then $S_1(\mathbb{B},j), S_2(\mathbb{B},j), \ldots, S_{K-1}(\mathbb{B},j)$
are proper subsets of $[n]$ and  $S_{K}(\mathbb{B},j)=[n]$.

Now we show $S_1(\mathbb{B},j), S_2(\mathbb{B},j), \ldots, S_K(\mathbb{B},j)$ are pairwise distinct.
Otherwise, there exist $u,v$ such that $1\leq u<v\leq K$ and $S_1(\mathbb{B},j), S_2(\mathbb{B},j), \ldots, S_u(\mathbb{B},j), \ldots, S_{v-1}(\mathbb{B},j)$  are pairwise distinct,
but $S_u(\mathbb{B},j)=S_v(\mathbb{B},j)$. Then by (\ref{eq22}), we have
$S_{u+t+s(v-u)}(\mathbb{B},j)=S_{u+t}(\mathbb{B},j)$ for any $t$ with $0\leq t< v-u$ and nonnegative integer $s$,
and thus there exist nonnegative integers $s_1, t_1$ with $0\leq t_1< v-u$ such that $K=u+t_1+s_1(v-u)$
and $S_{K}(\mathbb{B},j)=S_{u+t_1}(\mathbb{B},j)\not=[n],$ it is a contradiction.
We note that   $S_1(\mathbb{B},j), S_2(\mathbb{B},j), \ldots, S_K(\mathbb{B},j)$ are subset of $[n]$ and they are not empty set,
so $K\leq 2^n-1.$ Thus $r(m,n)\leq 2^n-1.$
\end{proof}

\begin{proposition}\label{pro49}
Let  $m\geq 2$, and  $\mathbb{A}=(a_{i_1\ldots i_m})$ be a nonnegative tensor with  order $m$ and  dimension $n$,
$\mathbb{B}=(b_{i_1\ldots i_mi_{m+1}})$ be a nonnegative tensor  with  order $m+1$ and  dimension $n$ where
$$b_{i_1\ldots i_mi_{m+1}}=\left\{\begin{array}{ll}
                                                   a_{i_1\ldots i_m}, & \mbox{ if } i_{m+1}=i_m; \\
                                                   0, & \mbox{ otherwise. }
                                                 \end{array}\right.$$
Then we have

{\rm (i). } There exists some $j\in [n]$ satisfying $\mathbb{A}$ is $j$-primitive
if and only if there exists some $j\in [n]$ satisfying $\mathbb{B}$ is $j$-primitive,
 and $\gamma_j(\mathbb{B})=\gamma_j(\mathbb{A}).$

 {\rm (ii). } $\mathbb{A}$ is primitive if and only if $\mathbb{B}$ is primitive, and thus $\gamma(\mathbb{B})=\gamma(\mathbb{A}).$
\end{proposition}
\begin{proof}
Firstly, we show  $S_k(\mathbb{B},j)=S_k(\mathbb{A},j)$ for $k\geq 1$ by induction on $k$. Clearly,
  $$S_1(\mathbb{B},j)=\{u\in[n]|b_{uj\ldots j}>0\}=\{u\in[n]|a_{uj\ldots j}>0\}=S_1(\mathbb{A},j).$$
Now we assume that $S_k(\mathbb{B},j)=S_k(\mathbb{A},j)$ for $k\geq 1$. Then by (\ref{eq22}), we have

$S_{k+1}(\mathbb{B},j)=\{u\in[n]|\mbox{ there exist } i_2, \ldots, i_m, i_{m+1}\in S_k(\mathbb{B},j) \mbox{ and }
b_{ui_2\ldots i_mi_{m+1}}>0\}$

\noindent\hskip2.5cm $=\{u\in[n]|\mbox{ there exist } i_2, \ldots, i_m, i_{m+1}\in S_k(\mathbb{A},j) \mbox{ and }
b_{ui_2\ldots i_mi_{m+1}}>0\}$

\noindent\hskip2.5cm $=\{u\in[n]|\mbox{ there exist } i_2, \ldots, i_m\in S_k(\mathbb{A},j) \mbox{ and }
b_{ui_2\ldots i_mi_{m}}>0\}$

\hskip1.9cm $=\{u\in[n]|\mbox{ there exist } i_2, \ldots, i_m\in S_k(\mathbb{A},j) \mbox{ and }
a_{ui_2\ldots i_m}>0\}$

\noindent\hskip2.5cm $=S_{k+1}(\mathbb{A},j).$

Then $\mathbb{B}$ is $j$-primitive if and only if $\mathbb{A}$ is $j$-primitive,
and  $\gamma_j(\mathbb{B})=\gamma_j(\mathbb{A}).$

 It is obvious that $\mathbb{A}$ is primitive if and only if for all $j(\in [n])$, $\mathbb{A}$ is $j$-primitive.
 Thus (ii) holds by (i).
 \end{proof}


By Proposition \ref{pro49}, we obtain the following results  and propose some questions immediately.

\begin{proposition}\label{pro410}
Let $l$ be a large positive integer. Then

{\rm (i). }  $r(2,n)\leq r(3,n)\leq r(4,n)\leq \ldots \leq r(l,n)\leq r(l+1,n)$.

{\rm (ii).  } $R_j(2,n)\subseteq R_j(3,n)\subseteq R_j(4,n)\subseteq \ldots \subseteq R_j(l,n)\subseteq R_j(l+1,n)$.
\end{proposition}

 \begin{problem}\label{pro411}
Let $m(\geq 3), n(\geq 2)$ be positive integers. Then $r(m,n)=?$ Can you characterize the extremal tensor?
\end{problem}

\begin{problem}\label{pro412}
Does there exist gaps in $R_j(m,n)$?
\end{problem}

In \cite{Sh12}, Shao proposed the concept of strongly primitive.

\begin{definition}\label{defn413}{\rm (\cite{Sh12}, Definition 4.3)}
Let $\mathbb{A}$ be a nonnegative tensor with order $m$ and dimension $n$. If there exists some positive
integer $k$ such that $\mathbb{A}^k > 0$ is a positive tensor, then $\mathbb{A}$ is called strongly primitive,
and the smallest such $k$ is called the strongly primitive degree of $\mathbb{A}$.
\end{definition}

Let $\mathbb{A}=(a_{i_1i_2\ldots i_m})$ be a nonnegative tensor with order $m$ and dimension $n$.
It is clear that if $\mathbb{A}$ is strongly primitive, then $\mathbb{A}$ is primitive.
In fact, it is obvious that in the matrix case $(m = 2)$
 a nonnegative matrix $A$ is primitive if and only if $A$ is strongly primitive,
but in the case $m\geq 3$, Shao gave an example to show these two concepts are not equivalent(\cite{Sh12}).
Till now, there are no further research on strongly primitive. 
For convenience,  let $\eta(\mathbb{A})$ be the strongly primitive degree of $\mathbb{A}$.

\begin{proposition}\label{pro414}
Let $\mathbb{A}=(a_{i_1i_2\ldots i_m})$ be a nonnegative strongly primitive tensor with order $m$ and dimension $n$.
Then for any $\alpha\in [n]^{m-1}$, there exists some $i\in [n]$ such that $a_{i\alpha}>0$.
\end{proposition}
\begin{proof}
Without loss of generality, we assume   the strongly primitive degree of $\mathbb{A}$, denoted by $\eta(\mathbb{A})$, is $k$.
Then $A^k$ be a positive tensor with order $(m-1)^k+1$ and dimension $n$. Therefore for any $i_1\in [n]$
and any $\alpha_2, \ldots, \alpha_t\in [n]^{m-1}$ where $t=(m-1)^{k-1}$, we can complete the proof by the following equation.

\begin{equation}\label{eq42}
(A^k)_{i_1\alpha_2\ldots \alpha_t}=\sum\limits_{i_2,\ldots, i_t=1}^{n} (A^{k-1})_{i_1i_2\ldots i_t} a_{i_2\alpha_2}\ldots a_{i_t\alpha_t}>0.
\end{equation}
\end{proof}

\begin{example}\label{ex415}
Let $m=n=3$, $\mathbb{A}=(a_{i_1i_2i_3})$ be a nonnegative tensor with order $m$ and dimension $n$, 
where $a_{111}=a_{222}=a_{333}=a_{233}=a_{311}=0$ and other $a_{i_1i_2i_3}=1$. Then $\eta(\mathbb{A})=4.$
\end{example}
\begin{proof}
Firstly, we show that $\mathbb{A}^2=(a_{i_1i_2i_3i_4i_5}^{(2)})$ is a nonnegative tensor with order $5$ and dimension $3$, 
and $a_{i_1j_2j_3j_4j_5}>0$ except for $a_{13333}=a_{21111}=a_{33333}=0$.
We complete the proof by the following four cases and  
$a_{i_1j_2j_3j_4j_5}^{(2)}=\sum\limits_{i_2,i_3=1}^{3}a_{i_1i_2i_3}a_{i_2j_2j_3}a_{i_3j_4j_5}.$

{\bf Case 1: } $j_2\not=j_3$ and $j_4\not=j_5$.

Then $a_{i_2j_2j_3}=a_{i_3j_4j_5}=1$ for any $i_2, i_3\in \{1,2,3\}$. Thus 
$a_{i_1j_2j_3j_4j_5}^{(2)}=\sum\limits_{i_2,i_3=1}^{3}a_{i_1i_2i_3}>0.$

{\bf Case 2: } $j_2=j_3$ and $j_4\not=j_5$.

Then $a_{i_3j_4j_5}=1$ for any $i_3\in \{1,2,3\}$. Thus
$$a_{i_1j_2j_3j_4j_5}^{(2)}=\sum\limits_{i_2,i_3=1}^{3}a_{i_1i_2i_3}a_{i_2j_2j_3}\geq \left\{\begin{array}{ll}
                                                                                           a_{i_123}a_{211}, & \mbox{if } j_2j_3=11; \\
                                                                                           a_{i_113}a_{122}+a_{i_131}a_{322}, & \mbox{if } j_2j_3=22; \\
                                                                                           a_{i_113}a_{133}, & \mbox{if } j_2j_3=33 \\
                                                                                         \end{array}\right.
$$
\hskip3.8cm $>0.$

{\bf Case 3: } $j_2\not=j_3$ and $j_4=j_5$.

The proof is similar to the proof of Case 2.

{\bf Case 4: } $j_2=j_3$ and $j_4=j_5$.

In this case, we know only  $a_{13333}=a_{21111}=a_{33333}=0$ by direct computation.

Similarly,  we can show that $\mathbb{A}^3=(a_{i_1i_2\ldots i_9}^{(3)})$ is a almost positive tensor with order $9$ and dimension $3$
except for $a_{233333333}^{(3)}=0$ by (\ref{eq42}), but $\mathbb{A}^4, \mathbb{A}^5, \ldots$ are positive tensors. 
Thus combining the above cases, we know $\eta(\mathbb{A})=4.$
\end{proof}

\begin{problem}\label{pro416}
Can we define and  study the strongly primitive degree, the strongly exponent set,
the $j$-strongly primitive of strongly primitive tensors and so on?
\end{problem}



\end{document}